\documentclass[11pt]{article}

\usepackage{amssymb}
\usepackage{amsfonts}
\usepackage{graphicx}
\usepackage{latexsym}
\usepackage{amsmath}
\usepackage{changepage}
\usepackage{color}
\usepackage{latexsym}
\usepackage{tablefootnote}
\usepackage[shortlabels]{enumitem}
\usepackage{epsfig}
\usepackage{multirow}
\usepackage{float}
\usepackage{array}
\usepackage{bbm}
\usepackage{dsfont}
\usepackage{relsize}
\usepackage{subcaption}

\allowdisplaybreaks

\newtheorem{thm}{Theorem}[section]
\newtheorem{theorem}[thm]{Theorem}

\newtheorem{lemma}[thm]{Lemma}

\newtheorem{proposition}[thm]{Proposition}

\newtheorem{remark}[thm]{Remark}

\newcommand{\mb}[1]{\mbox{\boldmath $#1$}}

\begin{document}
\title{A First Order Algorithm on an Optimization Problem with Improved Convergence when Problem is Convex }

\author{{\bf{Chee-Khian Sim}}\footnote{Email address:
chee-khian.sim@port.ac.uk}
 \\ School of Mathematics and Physics\\
University of Portsmouth \\ Lion Gate Building, Lion Terrace \\ Portsmouth PO1 3HF
}

\date{14 August 2025}
\maketitle

\begin{abstract}
We propose a first order algorithm, a modified version of FISTA, to solve an optimization problem with an objective function that is the sum of a possibly nonconvex function, with Lipschitz continuous gradient, and a convex function which can be nonsmooth.  The algorithm is shown to have an iteration complexity of $\mathcal{O}(\epsilon^{-2})$ to find an $\epsilon$-approximate solution to the problem, and this complexity improves to $\mathcal{O}(\epsilon^{-2/3})$ when the objective function turns out to be convex.   We further provide asymptotic convergence rate  for the algorithm of worst case $o(\epsilon^{-2})$ iterations to find an $\epsilon$-approximate solution to the problem, with worst case $o(\epsilon^{-2/3})$ iterations  when its objective function is convex.

\vspace{10pt}

\end{abstract}

\section{Introduction}\label{sec:Intro}

First order methods are frequently used to solve optimization problems, especially when these problems are large scale \cite{Beck2}.  Fast iterative shrinkage thresholding algorithm (FISTA), a first order method, is proposed in \cite{Beck} to solve a convex minimization problem whose objective function is the sum of two convex functions, where one function can be nonsmooth; see also \cite{Nesterov4, Nesterov6, Nesterov5}.   Algorithms, including first order algorithms,  to solve minimization problems whose objective functions are nonconvex are also studied in the literature (see for example \cite{Cartis}).  The objective function can be a sum of two functions, where one function is nonconvex, has Lipschitz continuous gradient, and the other function is convex, but nonsmooth.  Solving this minimization problem has application in areas such as training a neural network, as discussed for example in  \cite{Goodfellow}.   Existing works that study first order algorithms on nonconvex optimization problems include \cite{Carmon2,Ghadimi,Ghadimi2,Kong,Li,Li2,Liang,Liang2,Liang3,Yao2}.  

\vspace{10pt}

\noindent In the works mentioned in the previous paragraph, such as \cite{Ghadimi, Liang3}, the iteration complexity to find an $\epsilon$-approximate solution to the nonconvex optimization problem, whose objective function is the sum of two functions, where one function can be nonconvex, has Lipschitz continuous gradient, and the other function is convex, but nonsmooth, is shown to be $\mathcal{O}(\epsilon^{-2})$.  This is also the iteration complexity found for the first order algorithm, which is a modified version of FISTA, proposed in this paper on the optimization problem.   Our algorithm also has the added feature that when the objective function to the problem turns out to be convex, it is able to ``detect'' this and its iteration complexity improves to $\mathcal{O}(\epsilon^{-2/3})$.  In \cite{Ghadimi}, the iteration complexity of the algorithm there also improves to $\mathcal{O}(\epsilon^{-2/3})$ when the objective function is convex, but the parameter inputs to the algorithm need to be adjusted accordingly.  In \cite{Ghadimi2}, the authors improve on \cite{Ghadimi} by introducing a parameter free accelerated gradient method to solve nonconvex optimization problems, although the authors only show optimal complexity on function values when the objective function is convex.   In \cite{Liang3}, the accelerated composite gradient (ACG) variant introduced in the paper reduces to FISTA when the objective function is convex.   The algorithm in this paper however is different from that in \cite{Liang3} in the way nonconvexity is handled, with different analysis as well.  As with the algorithm in \cite{Liang3}, our algorithm performs exactly one resolvent evaluation of $h$ in each iteration, unlike  papers found in the literature such as \cite{Ghadimi}, where its ACG variant requires two resolvent evaluations of $h$ in each iteration.   We note that in \cite{Shi},  Nesterov's accelerated gradient method is shown to have the iteration complexity of $\mathcal{O}(\epsilon^{-2/3})$ when the objective function consists only of a single  convex function with Lipschitz continuous gradient.  The approach for analysis in the paper is via high-resolution differential equations.

\vspace{10pt}

\noindent In this paper, we also provide asymptotic convergence rate\footnote{The study of asymptotic convergence is not new.  For example, local superlinear convergence of interior point methods has been analyzed by researchers \cite{Kojima, Luo, Nesterov2, Potra1, Sim1, Wright}.} results for our algorithm, both when the objective function is nonconvex and convex.  Our asymptotic convergence rate for the convex case is also found in \cite{Chen} through high-resolution differential equations.  In our case, we analyze our algorithm directly to obtain the result.   In the recent published paper \cite{Gratton}, the authors show that it takes $o(\epsilon^{-2})$ iterations to find an $\epsilon$-approximate solution to a nonconvex optimization problem using steepest descent algorithm and other linesearch methods.  In this paper, we add further to this by showing asymptotic convergence rate of $o(n^{-1/2})$ for our algorithm, a first order algorithm of the accelerated type\footnote{Examples of first order algorithms of the accelerated type on nonconvex optimization problems include accelerated composite gradient variants \cite{Liang3} and accelerated proximal gradient methods as found in \cite{Li}.}, when we solve the nonconvex optimization problem we considered in the paper, under the assumption of convergence of iterates.   To the best of our knowledge, this is the first time asymptotic convergence rate result is presented on first order algorithms of the accelerated type when solving nonconvex optimization problems.


\vspace{10pt}

\noindent In the next section, Section \ref{sec:CNO}, we introduce the optimization problem we are solving in this paper, and also present the first order algorithm we are proposing to solve the problem.   In Section \ref{sec:technical}, we state and prove results that are needed for the iteration complexities and asymptotic convergence rate results in its next section and associated subsection.  We conclude the paper with Section \ref{sec:conclusion}.

\subsection{Notations and facts}\label{subsec:notationsfacts}

Given $\Omega$ a closed convex set of $\Re^n$, we denote the standard projection map from $\Re^n$ onto $\Omega$ by $P_\Omega(\cdot)$.  We have the following result on $P_{\Omega}$:
\begin{eqnarray}\label{ineq:projection}
\| P_\Omega(x) - y \| \leq \| x - y \|,
\end{eqnarray}
where $x \in \Re^n$ and $y \in \Omega$.

\vspace{10pt}

\noindent The following is a simple fact used in this paper:
\begin{eqnarray}\label{eq:triangle}
\| a - b \|^2 - 2 \langle a - b, a - c \rangle + \| a - c \|^2 = \| b - c \|^2,
\end{eqnarray} 
where $a, b, c \in \Re^n$.

\vspace{10pt}

\noindent Given functions $f: \mathcal{S} \rightarrow \mathbb{Z}_{++}$ and $g: \mathcal{S} \rightarrow \Re_{++}$, where $\mathcal{S}$ is an arbitrary set. For a subset $\hat{\mathcal{S}} \subseteq \mathcal{S}$, we write $f(w) = \mathcal{O}(g(w))$ for all $w \in \hat{\mathcal{S}}$ to mean that $f(w) \leq M g(w)$ for all $w \in \hat{\mathcal{S}}$, where $M > 0$ is a positive constant.  For $\hat{\mathcal{S}} = (0, w_0]$, where $w_0 > 0$, we write $f(w) = o(g(w))$ to mean that $f(w)/g(w) \rightarrow 0$, as $w \rightarrow 0$.

\section{An Optimization Problem and Modified FISTA}\label{sec:CNO} 

We consider the following optimization problem:
\begin{eqnarray}\label{CNO}
\min_{y \in \Re^n}\phi(y) := f(y) + h(y),
\end{eqnarray}
where $f$ can be nonconvex on $\Omega \subset \Re^n$, has Lipschitz continuous gradient with Lipschitz constant $L$, and $\Omega$ is a closed convex set in $\Re^n$.  Hence,  we have
\begin{eqnarray*}
\| \nabla f(u_1) - \nabla f(u_2) \| \leq L \|u_1 - u_2\|,\ \forall\ u_1, u_2 \in \Omega.
\end{eqnarray*}
As a consequence, the following holds: 
\begin{eqnarray}\label{ineq:lipschitzconsequence}
| f(u_1) - l_f(u_1; u_2) | \leq \frac{L}{2} \| u_1 - u_2 \|^2, \ \forall\ u_1, u_2 \in \Omega,
\end{eqnarray}
where
\begin{eqnarray*}
\ell_f(u_1;u_2) := f(u_2) + \langle \nabla f(u_2), u_1 - u_2 \rangle.
\end{eqnarray*}
In particular,
\begin{eqnarray}\label{ineq:upperlipschitzconsequence}
f(u_1) - l_f(u_1; u_2)  \leq \frac{L}{2} \| u_1 - u_2 \|^2, \ \forall\ u_1, u_2 \in \Omega.
\end{eqnarray}
Furthermore, $h$ is a proper lower semi-continuous convex function, which can be non-smooth, with ${\rm{dom}}\ h \subset \Re^n$ closed and bounded.  We assume that ${\rm{dom}}\ h \subseteq \Omega$, and $\| y \| \leq C$ for $y \in {\rm{dom}}\ h$, where $C$ is some positive constant.  An example of $h$ is the indicator function of a closed and bounded convex set.
%
%
%
\vspace{10pt}

\noindent Observe that from (\ref{ineq:lipschitzconsequence}), we have
\begin{eqnarray}\label{ineq:lowerlipschitzconsequence}
-\frac{L}{2} \| u_1 - u_2 \|^2 \leq f(u_1) - l_f(u_1; u_2), \ \forall\ u_1, u_2 \in \Omega.
\end{eqnarray}
Therefore, there exists $m \geq 0$ such that
\begin{eqnarray}\label{ineq:lowerlipschitzconsequence2}
-\frac{m}{2} \| u_1 - u_2 \|^2 \leq f(u_1) - l_f(u_1; u_2), \ \forall\ u_1, u_2 \in \Omega.
\end{eqnarray}
An example of such $m$ is $L$ as seen from (\ref{ineq:lowerlipschitzconsequence}).  Let $\underline{m} \geq 0$ be the smallest $m$ such that (\ref{ineq:lowerlipschitzconsequence2}) holds.  It is clear that $\underline{m} \leq L$.  Furthermore,  $\underline{m} = 0$ if and only if $f$ is convex on $\Omega$.

\vspace{10pt}

\noindent Note that there exists an optimal solution to Problem (\ref{CNO}), which we denote by $y^\ast$, since ${\rm{dom}}\ h$ is closed and bounded.   A necessary condition for $y \in {\rm{dom}} \ h$ to be an optimal solution of Problem (\ref{CNO}) is $0 \in \nabla f({y}) + \partial h(y)$.  Motivated by this condition, we have the notion of an $\epsilon$-approximate solution\footnote{This notion of closeness to optimality is employed in \cite{Kong,Kong2,Liang3} and other papers that can be found in the literature.} of Problem (\ref{CNO}), which is an $(y, v)$ which satisfies
\begin{eqnarray}\label{approximatesolution}
v \in \nabla f(y) + \partial h(y), \quad \| v \| \leq \epsilon,
\end{eqnarray}
where ${\epsilon} > 0$ is a given tolerance.  

\vspace{10pt}

\noindent We now propose a first order algorithm, which we called Algorithm mFISTA, to find an ${\epsilon}$-approximate solution to Problem (\ref{CNO}), for a given tolerence ${\epsilon} > 0$.   This algorithm is a modification of FISTA for convex optimization as found in \cite{Beck} (see also \cite{Nesterov4}) to the setting when $f$ in Problem (\ref{CNO}) can be nonconvex.


\noindent\rule[0.5ex]{1\columnwidth}{1pt}

{\centerline{\bf{{Algorithm mFISTA}}}}

\noindent\rule[0.5ex]{1\columnwidth}{1pt}

\noindent {\bf{Initialization}}: Let  tolerance $\epsilon > 0$, 
	initial point $x_{1} = y_{0}  \in {\rm{dom}} \ h$ and $a_{0} = 1$ be given.   Set $k = 1$ with $L_k = 0$.  
	
%

\vspace{10pt}

\noindent {\bf{Step $\mb{1}$}}.   Set
		\begin{eqnarray}\label{eq:funfk}
	 	 f_{k}(y) = f(y) + \frac{L_k}{2} \| y - y_{k-1} \|^2.
	 	 \end{eqnarray}
	 	 
\vspace{5pt}

\noindent {\bf{Step $\mb{2}$}}.   Compute
		\begin{eqnarray}\label{eq:y_k}
		y_{k} = {\mbox{argmin}}_{y \in \Re^n} \left\{ l_{f_{k}}(y; x_{k}) + h(y) + 2L \|y - {x}_{k} \|^2 \right\}.
		\end{eqnarray}

\vspace{5pt}

\noindent {\bf{Step $\mb{3}$}}.   Find $a_{k} > 0$ such that 
		\begin{eqnarray}\label{eq:ak}
		a_{k} (a_{k} - 1) = a_{k-1}^2.
		\end{eqnarray}

\vspace{5pt}

\noindent {\bf{Step $\mb{4}$}}. Compute 
		\begin{eqnarray}
		v_{k} & = & \nabla f(y_{k}) - \nabla f(x_{k}) + L_k[y_{k-1} - x_{k}] +  4L [{x}_{k} - y_{k}], \label{eq:vk} \\
		{x}_{k+1} & = & P_\Omega \left( y_{k} + \frac{a_{k-1} - 1}{a_{k}} (y_{k} - y_{k-1}) \right).  \label{eq:hatxk}
		\end{eqnarray}

\vspace{5pt}

\noindent {\bf{Step $\mb{5}$}}.  Set
\begin{eqnarray*}
 A_k & = & \left\{ \begin{array}{ll}
 \displaystyle \frac{2( l_f(y_k; x_{k+1}) - f(y_k))}{\| y_k - x_{k+1} \|^2}  & , {\rm{if}} \  \| y_k - x_{k+1} \| \not= 0 \\
 0 & , {\rm{otherwise}} \end{array} \right..
 \end{eqnarray*}
 Compute  $L_{k+1} = \max \{ 0, A_k \}$.  Increase the value of $k$ by 1, and return to Step 1.   

\vspace{10pt}

\noindent {\bf{Termination:}} If $\|v_{k} \| \leq \epsilon$, then output  $(y_{k},v_{k})$ and {\bf{exit}}.


\noindent \rule[0.5ex]{1\columnwidth}{1pt}

%
%
%

\vspace{10pt}

\noindent Note that in Step 2 of the algorithm, $y_k$ obtained is in ${\rm{dom}} \ h$.  It is easy to see that $v_{k}$ obtained in Step 4 of the algorithm is in $\nabla f(y_{k}) + \partial h(y_{k})$.  Therefore, if $\| v_{k} \| \leq \epsilon$ when the algorithm terminates, then $(y_{k},v_{k})$ is an $\epsilon$-approximate solution to Problem (\ref{CNO}).   Furthermore, observe that if $\Omega = \Re^n$, then there is no need to perform a projection in (\ref{eq:hatxk}) to obtain $x_{k+1}$, which is not necessarily in ${\rm{dom}}\ h$.

\begin{remark}\label{rem:mFISTA}
We observe that in Algorithm mFISTA,  $0 \leq L_k \leq \underline{m}\ (\leq L)$ by  (\ref{ineq:lowerlipschitzconsequence2}).  When $f$ is convex, then $L_k = \underline{m} = 0$ for all $k$.  It is worth noting that $\underline{m}$ is not an input parameter to the algorithm, and the algorithm is able to ``detect'' when $\underline{m} = 0$ (that is, $f$ is convex), resulting in an improved iteration complexity, as shown in Theorem \ref{thm:main}, in this case.
\end{remark}

\noindent We end this section with the following result on the sequence $\{ a_{k} \}$:
\begin{proposition}\label{prop:ak}
We have for $k \geq 0$,
\begin{eqnarray}\label{ineq:ak}
a_{k} \geq \frac{k+1}{4}
\end{eqnarray}
and 
\begin{eqnarray}\label{ineq:ak2}
a_{k} \leq 1 + k.
\end{eqnarray}
Hence, $\displaystyle \left\{ \frac{a_{k+1}}{a_{k}} \right\}$ is bounded.  In fact, $\displaystyle \frac{a_{k+1}}{a_{k}} \rightarrow 1$ as $k \rightarrow \infty$.
\end{proposition}
\begin{proof}
Note that from (\ref{eq:ak}), we have
\begin{eqnarray}\label{eq:ak2}
a_{i} = \frac{1 + \sqrt{1 + 4a_{i-1}^2}}{2}.
\end{eqnarray}
Hence, 
\begin{eqnarray*}
a_{i}  = \frac{1 + \sqrt{1 + 4 a_{i-1}^2}}{2} \geq \frac{1}{2} + a_{i-1}.
\end{eqnarray*}
Summing the above from $i = 1$ to $k$, we have
\begin{eqnarray*}
a_{k} \geq \frac{k}{2} + a_{0} \geq \frac{k}{2} \geq \frac{k+1}{4}.
\end{eqnarray*}
On the other hand,
\begin{eqnarray}\label{ineq:ak3}
a_{i} = \frac{1 + \sqrt{1 +  4 a_{i-1}^2}}{2} \leq \frac{ 1 + 1 + \sqrt{4a_{i-1}^2}}{2}  
=  1 + a_{i-1}.  
\end{eqnarray}
Hence, summing the above from $i = 1$ to $k$, we have
\begin{eqnarray*}
a_{k} \leq a_0 + k = 1 + k.
\end{eqnarray*}
Finally, boundedness of $\displaystyle \left\{ \frac{a_{k+1}}{a_{k}} \right\}$ follows from (\ref{ineq:ak}) and (\ref{ineq:ak2}).    It further follows from (\ref{eq:ak2}) and (\ref{ineq:ak}) that $\displaystyle \frac{a_{k+1}}{a_{k}} \rightarrow 1$ as $k \rightarrow \infty$.
\end{proof}

\section{Technical Preliminaries}\label{sec:technical}

\begin{proposition}\label{prop:upperlowerboundfk}
We have for $k \geq 1$ and $y \in \Omega$,
\begin{eqnarray*}
\hat{L}_k \| y - x_k \|^2 \leq f_{k}(y) - l_{f_{k}}(y; x_k) \leq L \| y - x_k \|^2,
\end{eqnarray*}
where
\begin{eqnarray*}
\hat{L}_k = \left\{ \begin{array}{cl}
			0 & ,{\rm{if}}\ y = y_{k-1} \\
			\displaystyle \frac{L_k - \underline{m}}{2} & ,{\rm{otherwise}}
			\end{array} \right..
\end{eqnarray*}
\end{proposition}
\begin{proof}
For $y \in \Omega$, 
\begin{eqnarray*}
f_{k}(y) - l_{f_{k}}(y; x_k) & = & f(y) + \frac{L_k}{2}\|y - y_{k-1} \|^2 - f(x_k) - \frac{L_k}{2} \|x_k - y_{k-1}\|^2 - \\
&  & \langle \nabla f(x_k) + L_k (x_k - y_{k-1}), y - x_k \rangle \\
& = & f(y) - f(x_k) - \langle \nabla f(x_k), y - x_k \rangle + \\
&    & \frac{L_k}{2} [\| y - y_{k-1} \|^2 + 2 \langle x_k - y_{k-1}, x_k - y \rangle - \| x_k - y_{k-1} \|^2] \\
& = & f(y) - l_f(y;x_k) + \frac{L_k}{2} \| y - x_k \|^2,
\end{eqnarray*}
where the last equality follows from (\ref{eq:triangle}) and the definition of $l_f$.  We have by (\ref{ineq:upperlipschitzconsequence}) and Remark \ref{rem:mFISTA} that
\begin{eqnarray*}
 f(y) - l_f(y;x_k) + \frac{L_k}{2} \| y - x_k \|^2 \leq L \| y - x_k\|^2.
 \end{eqnarray*}
On the other hand, when $y = y_{k-1}$, by definition of $L_k$,
 \begin{eqnarray*}
f(y) - l_f(y;x_k) + \frac{L_k}{2} \| y - x_k \|^2 & = &  f(y_{k-1}) - l_f(y_{k-1};x_k) + \frac{L_k}{2} \| y_{k-1} - x_k \|^2 \\
& \geq &   f(y_{k-1}) - l_f(y_{k-1};x_k) +\frac{A_{k-1}}{2} \| y_{k-1} - x_k \|^2 \\
& = &  0.
 \end{eqnarray*} 
Otherwise, by (\ref{ineq:lowerlipschitzconsequence2}),
 \begin{eqnarray*}
f(y) - l_f(y;x_k) + \frac{L_k}{2} \| y - x_k \|^2  & \geq &   f(y) - l_f(y;x_k) + \frac{\underline{m}}{2} \| y - x_k \|^2 + \frac{L_k - \underline{m}}{2} \| y - x_k \|^2 \\
& \geq &  \frac{L_k - \underline{m}}{2} \| y - x_k \|^2 .
 \end{eqnarray*} 
\end{proof}

\vspace{10pt}

\noindent The above proposition leads to the following consequence:
\begin{proposition}\label{prop:convexityfk}
We have for $k \geq 1$ and  $y \in {\rm{dom}}\ h$,
\begin{eqnarray}
f_{k}(y_{k}) + h(y_{k}) - f_{k}(y) - h(y) + L\| y_{k} - {x}_{k} \|^2  & \leq &   2L [\| y - {x}_{k} \|^2 - \| y - y_{k} \|^2] - \nonumber \\
& &  \hat{L}_k  \| y - {x}_{k} \|^2.   \label{ineq:convexityfk}
\end{eqnarray}
\end{proposition}
\begin{proof}
By Proposition \ref{prop:upperlowerboundfk}, for $y \in {\rm{dom}}\ h$, we have
\begin{eqnarray*}
f_{k}(y) + h(y) + (2	L - \hat{L}_k) \| y - {x}_{k} \|^2 \geq l_{f_{k}}(y; x_{k}) + h(y) + 2L \| y - {x}_{k} \|^2.
\end{eqnarray*}
Now, since $y_{k}$ is the optimal solution to $l_{f_{k}}(y; x_{k}) + h(y) + 2L \| y - {x}_{k} \|^2$ over $y \in \Re^n$, we have
\begin{eqnarray*}
 l_{f_{k}}(y; x_{k}) + h(y) + 2L \| y - {x}_{k} \|^2 \geq  l_{f_{k}}(y_{k};x_{k}) + h(y_{k}) + 2L \|y_{k} - {x}_{k}\|^2 + 2L \| y - y_{k}\|^2.
\end{eqnarray*}
Hence,
\begin{eqnarray*}
&   & f_{k}(y) + h(y) + (2L - \hat{L}_k) \| y - {x}_{k} \|^2  \\
 & \geq & l_{f_{k}}(y_{k};x_{k}) + h(y_{k}) + 2L \|y_{k} - {x}_{k}\|^2 + 2L \| y - y_{k}\|^2 \\
& = & l_{f_{k}}(y_{k};x_{k}) + L \|y_{k} - x_{k} \|^2 + h(y_{k}) +  L \| y_{k} - {x}_{k} \|^2  +   2L \| y - y_{k}\|^2  \\
& \geq & f_{k}(y_{k}) + h(y_{k}) + L \| y_{k} - {x}_{k} \|^2 + 2L \| y - y_{k}\|^2,
\end{eqnarray*}
where the second inequality follows from Proposition \ref{prop:upperlowerboundfk}, and we prove (\ref{ineq:convexityfk}).
\end{proof}

%
%
%
\vspace{10pt}

\begin{proposition}\label{prop:technical}
We have for $k \geq 2$, 
\begin{eqnarray}
& & a_{k-1}(a_{k-1}-1)[\| y_{k-1} - {x}_{k} \|^2 - \| y_{k-1} - y_{k} \|^2] + a_{k-1}[\| {x}_{k} - y^\ast \|^2 - \| y_{k} - y^\ast \|^2] \nonumber \\
& \leq & \|a_{k-2}[ y_{k-1} - y_{k-2}]  + y_{k-2} - y^\ast \|^2 - \| a_{k-1}[ y_{k} - y_{k-1}]  + y_{k-1} - y^\ast \|^2.  \label{ineq:technical}
\end{eqnarray}
\end{proposition}
\begin{proof}
Note that
\begin{eqnarray*}
& & a_{k-1}(a_{k-1}-1)[\| y_{k-1} - {x}_{k} \|^2 - \| y_{k-1} - y_{k} \|^2] + a_{k-1}[\| {x}_{k} - y^\ast \|^2 - \| y_{k} - y^\ast \|^2] \\
& = &  a_{k-1}(a_{k-1}-1)[\| y_{k-1} - {x}_{k} \|^2 - \| (y_{k-1} - {x}_{k}) + ({x}_{k} - y_{k}) \|^2] + \\
&  & a_{k-1}[\| {x}_{k} - y^\ast \|^2 - \| (y_{k} - {x}_{k}) + ({x}_{k} - y^\ast) \|^2] \\
& = & a_{k-1}(a_{k-1}-1)[- \| {x}_{k} - y_{k} \|^2 - 2 \langle y_{k-1} - {x}_{k}, {x}_{k} - y_{k} \rangle] + \\
& &  a_{k-1}[-\| {x}_{k} - y_{k}\|^2 - 2\langle y_{k} - {x}_{k}, {x}_{k} - y^\ast \rangle] \\
& = & - a_{k-1}^2 \| {x}_{k} - y_{k} \|^2 - 2\langle a_{k-1}[{x}_{k} - y_{k}], y^\ast - {x}_{k} + (a_{k-1} - 1)(y_{k-1} - {x}_{k}) \rangle \\
& = & - a_{k-1}^2 \| y_{k} - {x}_{k} \|^2 -2 \langle a_{k-1}[y_{k} - {x}_{k}], a_{k-1} {x}_{k} - ((a_{k-1}-1)y_{k-1}+y^\ast) \rangle \\
& = & \| a_{k-1}{x}_{k} -  ((a_{k-1}-1)y_{k-1}+y^\ast) \|^2 - \| a_{k-1} y_{k} - ((a_{k-1}-1)y_{k-1} + y^\ast) \|^2,
 \end{eqnarray*}
 where the last equality follows from (\ref{eq:triangle}).  Now,
 \begin{eqnarray*}
& &  \left\| {x}_{k} - \left( \left(1 - \frac{1}{a_{k-1}}\right)y_{k-1} + \frac{1}{a_{k-1}}y^\ast \right) \right\| \\
& = & \left\| P_\Omega \left( y_{k-1} + \frac{a_{k-2} - 1}{a_{k-1}} (y_{k-1} - y_{k-2}) \right) - \left( \left(1 - \frac{1}{a_{k-1}}\right)y_{k-1} + \frac{1}{a_{k-1}}y^\ast \right) \right\| \\
& \leq & \left\|  \left( y_{k-1} + \frac{a_{k-2} - 1}{a_{k-1}} (y_{k-1} - y_{k-2}) \right) - \left( \left(1 - \frac{1}{a_{k-1}}\right)y_{k-1} + \frac{1}{a_{k-1}}y^\ast \right) \right\| \\
& = & \left\| \frac{a_{k-2}}{a_{k-1}}y_{k-1} - \frac{1}{a_{k-1}}[(a_{k-2}-1)y_{k-2} + y^\ast] \right\|,
 \end{eqnarray*} 
where the inequality follows from (\ref{ineq:projection}).  The result in the proposition then follows from the above derivations.
\end{proof}


\section{Iteration Convergence Rates of Algorithm mFISTA - Non-asymptotic and Asymptotic}\label{sec:mFISTA}

%
%
%

%

\noindent We first state and prove the following key inequality, which prepares us for the iteration complexity results for Algorithm mFISTA in Theorem \ref{thm:main}:
\begin{lemma}\label{lem:keyinequality}
We have for $k \geq 2$,
\begin{eqnarray}
& & a_{k-1}^2[\phi(y_{k}) - \phi(y^\ast)] - a_{k-2}^2 [\phi(y_{k-1}) - \phi(y^\ast)] + \frac{L_k}{2}a_{k-1}^2\|y_{k} - y_{k-1} \|^2 +  L a_{k-1}^2 \| y_{k} - {x}_{k} \|^2 \nonumber \\
& \leq &   2 L [\|a_{k-2}[ y_{k-1} - y_{k-2}]  + y_{k-2} - y^\ast \|^2 - \| a_{k-1}[ y_{k} - y_{k-1}]  + y_{k-1} - y^\ast \|^2] + \nonumber \\
&   & \frac{\underline{m} - L_k}{2} a_{k-1}  \| y^\ast - y_{k} \|^2 + \frac{L_k}{2} a_{k-1} \| y_{k-1} - y^\ast \|^2. \label{ineq:convexityfk7}
\end{eqnarray}
\end{lemma}
\begin{proof}
Letting $y = y_{k-1}$ and $y = y^\ast$ in (\ref{ineq:convexityfk}) lead to
\begin{eqnarray}
&   &  f_{k}(y_{k}) + h(y_{k}) - f_{k}(y_{k-1}) - h(y_{k-1}) +  L \| y_{k} - {x}_{k} \|^2 \nonumber \\
&  \leq &   2L[\| y_{k-1} - {x}_{k} \|^2 - \| y_{k-1} - y_{k} \|^2],  \label{ineq:convexityfk2} \\
&   &  f_{k}(y_{k}) + h(y_{k}) - f_{k}(y^\ast) - h(y^\ast) + L\| y_{k} - {x}_{k} \|^2 \nonumber \\
&  \leq &   2L[\| y^\ast - {x}_{k} \|^2 - \| y^\ast - y_{k} \|^2] + \frac{\underline{m} - L_k}{2}  \| y^\ast - {x}_{k} \|^2, \label{ineq:convexityfk3}
\end{eqnarray}
respectively.  Now, multiply (\ref{ineq:convexityfk2}) by $a_{k-1} - 1$ and add the resulting inequality to (\ref{ineq:convexityfk3})  gives rise to
\begin{eqnarray}
& & a_{k-1}[f_{k}(y_{k}) + h(y_{k}) - f_{k}(y^\ast) - h(y^\ast)] - (a_{k-1} - 1)[f_{k}(y_{k-1}) + h(y_{k-1}) -  f_{k}(y^\ast) - h(y^\ast)] + \nonumber \\
& & L a_{k-1} \| y_{k} - {x}_{k} \|^2  \nonumber \\
& \leq & 2L [(a_{k-1}-1)[\| y_{k-1} - {x}_{k} \|^2 - \| y_{k-1} - y_{k} \|^2] + 2L[\| y^\ast - {x}_{k} \|^2 - \| y^\ast - y_{k} \|^2]] + \nonumber \\
&   & \frac{\underline{m} - L_k}{2}  \| y^\ast - x_k \|^2. \label{ineq:convexityfk4}
\end{eqnarray}
Next, mutiply the inequality (\ref{ineq:convexityfk4}) by $a_{k-1}$ and noting (\ref{eq:ak}) then leads to
\begin{eqnarray}
& & a_{k-1}^2[f_{k}(y_{k}) + h(y_{k}) - f_{k}(y^\ast) - h(y^\ast)] - a_{k-2}^2[f_{k}(y_{k-1}) + h(y_{k-1}) - f_{k}(y^\ast) - h(y^\ast)] + \nonumber \\
& & L a_{k-1}^2 \| y_{k} - {x}_{k} \|^2 \nonumber \\
& \leq &  2 L a_{k-1}[(a_{k-1}-1)[\| y_{k-1} - {x}_{k} \|^2 - \| y_{k-1} - y_{k} \|^2] + [\| y^\ast - {x}_{k} \|^2 - \| y^\ast - y_{k} \|^2]] + \nonumber \\
&   & \frac{\underline{m} - L_k}{2} a_{k-1} \| y^\ast - x_k \|^2.  \label{ineq:convexityfk5}
\end{eqnarray}
Applying (\ref{ineq:technical}) in Proposition \ref{prop:technical} on the right hand side of the inequality in (\ref{ineq:convexityfk5}), we have
\begin{eqnarray}
& & a_{k-1}^2[f_{k}(y_{k}) + h(y_{k}) - f_{k}(y^\ast) - h(y^\ast)] - a_{k-2}^2[f_{k}(y_{k-1}) + h(y_{k-1}) - f_{k}(y^\ast) - h(y^\ast)] + \nonumber \\
& & La_{k-1}^2 \| y_{k} - {x}_{k} \|^2  \nonumber \\
& \leq &   2 L [\|a_{k-2}[ y_{k-1} - y_{k-2}]  + y_{k-2} - y^\ast \|^2 - \| a_{k-1}[ y_{k} - y_{k-1}]  + y_{k-1} - y^\ast \|^2] + \nonumber \\
&   & \frac{\underline{m} - L_k}{2} a_{k-1} \| y^\ast - y_{k} \|^2. \label{ineq:convexityfk6}
\end{eqnarray}
Now,
\begin{eqnarray*}
& & a_{k-1}^2[f_{k}(y_{k}) + h(y_{k}) - f_{k}(y^\ast) - h(y^\ast)] - a_{k-2}^2[f_{k}(y_{k-1}) + h(y_{k-1}) - f_{k}(y^\ast) - h(y^\ast)] \nonumber \\
& = & a_{k-1}^2[\phi(y_{k}) - \phi(y^\ast)]  + \frac{L_k}{2}a_{k-1}^2[\|y_{k} - y_{k-1} \|^2 - \| y^\ast - y_{k-1} \|^2]  - a_{k-2}^2[\phi(y_{k-1}) - \phi(y^\ast)] + \\
& & \frac{L_k}{2}a_{k-2}^2  \| y^\ast - y_{k-1} \|^2 \\
& = & a_{k-1}^2[\phi(y_{k}) - \phi(y^\ast)] - a_{k-2}^2 [\phi(y_{k-1}) - \phi(y^\ast)] + \frac{L_k}{2}a_{k-1}^2 \|y_{k} - y_{k-1} \|^2 - \\
&  & \frac{L_k}{2}( a_{k-2}^2 - a_{k-1}^2) \| y^\ast - y_{k-1} \|^2 \\
& = & a_{k-1}^2[\phi(y_{k} - \phi(y^\ast)] - a_{k-2}^2 [\phi(y_{k-1}) - \phi(y^\ast)] + \frac{L_k}{2}a_{k-1}^2  \|y_{k} - y_{k-1} \|^2 -  \frac{L_k}{2} a_{k-1} \| y_{k-1} - y^\ast \|^2,
\end{eqnarray*}
where the last equality follows from (\ref{eq:ak}).   From the above and (\ref{ineq:convexityfk6}), upon algebraic manipulations, we have
\begin{eqnarray*}
& & a_{k-1}^2[\phi(y_{k} - \phi(y^\ast)] - a_{k-2}^2 [\phi(y_{k-1}) - \phi(y^\ast)] + \frac{L_k}{2} a_{k-1}^2 \|y_{k} - y_{k-1} \|^2 +  La_{k-1}^2 \| y_{k} - {x}_{k} \|^2 \\
& \leq &   2 L [\|a_{k-2}[ y_{k-1} - y_{k-2}]  + y_{k-2} - y^\ast \|^2 - \| a_{k-1}[ y_{k} - y_{k-1}]  + y_{k-1} - y^\ast \|^2] + \\
&   & \frac{\underline{m} - L_k}{2}a_{k-1} \| y^\ast - y_{k} \|^2 + \frac{L_k}{2}a_{k-1} \| y_{k-1} - y^\ast \|^2, 
\end{eqnarray*}
and the lemma is proved.
\end{proof}

\vspace{10pt}

\noindent Summing the inequality in Lemma \ref{lem:keyinequality} from $k = 2$ to $n$ leads to
\begin{eqnarray}
& & a_{n-1}^2[\phi(y_{n}) - \phi(y^\ast)] - a_{0}^2 [\phi(y_{1}) - \phi(y^\ast)] + \frac{1}{2} \sum_{k=2}^n L_k a_{k-1}^2 \|y_{k} - y_{k-1} \|^2 + L \sum_{k=2}^n  a_{k-1}^2 \| y_{k} - {x}_{k} \|^2 \nonumber \\
& \leq &   2 L [\|a_{0}[ y_{1} - y_{0}]  + y_{0} - y^\ast \|^2 - \| a_{n-1}[ y_{n} - y_{n-1}]  + y_{n-1} - y^\ast \|^2] + \nonumber \\
&   &   \frac{1}{2} \sum_{k=2}^n (\underline{m} - L_k) a_{k-1} \| y^\ast - y_{k} \|^2 + \frac{1}{2} \sum_{k=2}^n L_k a_{k-1} \| y_{k-1} - y^\ast \|^2. \label{ineq:convexityfk8}
\end{eqnarray}

\noindent We are now ready for the main theorem in this paper:
\begin{theorem}\label{thm:main}
Given tolerance $\epsilon > 0$.  Algorithm mFISTA outputs an $\epsilon$-approximate solution to Problem (\ref{CNO}) in at most $\mathcal{O}(\epsilon^{-2})$ iterations.  In the case when $f$ in Problem (\ref{CNO}) is convex, the iteration complexity becomes $\mathcal{O}(\epsilon^{-2/3})$.   
\end{theorem}
\begin{proof}
From (\ref{eq:vk}), we have
\begin{eqnarray*}
\| v_{k} \| & =  & \| \nabla f(y_{k}) - \nabla f(x_{k}) + L_k[y_{k-1} - x_{k}] +  4L [{x}_{k} - y_{k}] \| \\
& \leq & (5L + L_k) \| x_k - y_k \| + L_k \| y_k - y_{k-1} \| \\
& \leq & 6L \| x_k - y_k \| + L_k \| y_k - y_{k-1} \|.
\end{eqnarray*}
Hence,
\begin{eqnarray}\label{ineq:main}
\min_{1 \leq k \leq n} \| v_k \| \leq 2 \sqrt{L}  \left( \min_{2 \leq k \leq n} \left( 36L \| x_k - y_k \|^2 + L_k \| y_k - y_{k-1} \|^2  \right) \right)^{1/2}.
\end{eqnarray} 
On the other hand, from (\ref{ineq:convexityfk8}), the following holds:
\begin{eqnarray}
&  & \frac{1}{36} \left( \min_{2 \leq k \leq n}  \left( 36L \| x_k - y_k \|^2 + L_k \| y_k - y_{k-1} \|^2 \right) \right) \sum_{k=2}^{n} a_{k-1}^2  \nonumber \\
& \leq &  2 L \|a_{0}[ y_{1} - y_{0}]  + y_{0} - y^\ast \|^2 + \frac{\underline{m}}{2} \sum_{k=2}^n a_{k-1} \| y^\ast - y_{k} \|^2 + \frac{1}{2} \sum_{k=2}^n L_k a_{k-1} \| y_{k-1} - y^\ast \|^2 \nonumber \\
& \leq & 2 L\|a_{0}[ y_{1} - y_{0}]  + y_{0} - y^\ast \|^2 + 2C^2\left( \underline{m} +  \max_{2 \leq k \leq n} L_k \right) \sum_{k=2}^n a_{k-1}, \label{ineq:main2}
\end{eqnarray}
where the last inequality follows from $\| y \| \leq C$ for $y \in {\rm{dom}}\ h$.  Note that by (\ref{ineq:ak}) and (\ref{ineq:ak2}),
\begin{eqnarray*}
\frac{1}{16} \sum_{k=2}^n k^2 \leq  \sum_{k=2}^{n} a_{k-1}^2 , \quad \sum_{k=2}^n a_{k-1} \leq \sum_{k=2}^n k.
\end{eqnarray*}
Using the above two inequalities, (\ref{ineq:main2}) implies that
\begin{eqnarray}
&  & \frac{1}{576}  \left(\min_{2 \leq k \leq n} \left(36L \| x_k - y_k \|^2 + L_k \| y_k - y_{k-1} \|^2 \right)\right) \sum_{k=2}^{n} k^2 \nonumber \\
& \leq & 2 L \|a_{0}[ y_{1} - y_{0}]  + y_{0} - y^\ast \|^2 + 2C^2\left( \underline{m} +  \max_{2 \leq k \leq n} L_k \right)  \sum_{k=2}^n k. \label{ineq:main3}
\end{eqnarray}
When $f$ is convex, $\underline{m} = 0$ and $\max_{2 \leq k \leq n} L_k  = 0$, hence (\ref{ineq:main3}) becomes
\begin{eqnarray}
\frac{1}{576}  \left( \min_{2 \leq k \leq n} \left(36L \| x_k - y_k \|^2 + L_k \| y_k - y_{k-1} \|^2 \right)\right) \sum_{k=2}^{n} k^2 \leq 2 L \|a_{0}[ y_{1} - y_{0}]  + y_{0} - y^\ast \|^2. \label{ineq:main4}
\end{eqnarray}
From (\ref{ineq:main}) and (\ref{ineq:main3}), we deduce that
\begin{eqnarray*}
\min_{1 \leq k \leq n} \| v_k \| = \mathcal{O}(n^{-1/2}),
\end{eqnarray*}
and when $f$ is convex, from (\ref{ineq:main}) and (\ref{ineq:main4}), we deduce that
\begin{eqnarray*}
\min_{1 \leq k \leq n} \| v_k \| = \mathcal{O}(n^{-3/2}).
\end{eqnarray*}
The theorem is hence proved.
\end{proof}

\begin{remark}\label{rem:convexityoptimalfunction}
It is easy to see from (\ref{ineq:convexityfk8}), where we have $\underline{m} = L_k = 0$ when $f$ is convex, that we can obtain the optimal complexity on function values of $\mathcal{O}(\epsilon^{-1/2})$ using Algorithm mFISTA to solve Problem (\ref{CNO}).  That is, it takes a worst case $\mathcal{O}(\epsilon^{-1/2})$ iterations using the algorithm to have $\phi(y_n) - \phi(y^\ast) \leq \epsilon$, when $f$ is convex.
\end{remark}

\subsection{Further Discussions}\label{subsec:discussion}

In this subsection, we analyze the convergence behavior of Algorithm mFISTA further.  

\vspace{10pt}

\noindent We first consider the case when $f$ is convex.  It turns out that we can get from the iteration complexity of Algorithm mFISTA of $O(\epsilon^{-2/3})$ to an asymptotic $o(\epsilon^{-2/3})$ iterations to find an $\epsilon$-approximate solution to Problem (\ref{CNO}) in this case.   

\vspace{10pt}

\noindent When $f$ is convex, it holds that $\underline{m} = L_k = 0$ for all $k$. Hence, (\ref{ineq:convexityfk7}) becomes
\begin{eqnarray}
& & a_{k-1}^2[\phi(y_{k}) - \phi(y^\ast)] - a_{k-2}^2 [\phi(y_{k-1}) - \phi(y^\ast)] + L a_{k-1}^2 \| y_{k} - {x}_{k} \|^2  \nonumber \\
& \leq &   2 L [\|a_{k-2}[ y_{k-1} - y_{k-2}]  + y_{k-2} - y^\ast \|^2 - \| a_{k-1}[ y_{k} - y_{k-1}]  + y_{k-1} - y^\ast \|^2].  \label{ineq:convexityfk9}
\end{eqnarray}
We see from (\ref{ineq:convexityfk9}) that $\{a_{k-1}^2[\phi(y_k) - \phi(y^\ast)] + 2L \|a_{k-1}[y_k - y_{k-1}] + y_{k-1} - y^\ast \|^2 \}$ is a nonincreasing sequence, and furthemore each term in the sequence nonnegative.  Hence, the sequence is convergent.   

\vspace{10pt}

\noindent By summing (\ref{ineq:convexityfk9}) from $k = n+1$ to $2n$, we have
\begin{eqnarray}
& & L \left( \min_{n+1 \leq k \leq 2n}  \| y_k - x_k \|^2  \right) \sum_{k=n+1}^{2n} a_{k-1}^2 \leq L \sum_{k=n+1}^{2n} a_{k-1}^2 \| y_k - x_k \|^2 \nonumber \\
& \leq   & (a_{n-1}^2[\phi(y_{n}) - \phi(y^\ast)] + 2L\| a_{n-1}[y_{n} - y_{n-1}] + y_{n-1} - y^\ast \|^2) - \nonumber \\
& &  (a_{2n - 1}^2 [ \phi(y_{2n}) - \phi(y^\ast)] + 2L\| a_{2n -1}[y_{2n} - y_{2n - 1}] + y_{2n - 1} - y^\ast \|^2). \label{ineq:convexityfk10}
\end{eqnarray}
We are then led to the following theorem:
\begin{theorem}\label{thm:main2}
When $f$ in Problem (\ref{CNO}) is convex, Algorithm mFISTA finds an $\epsilon$-approximate solution to the problem in at most  $o(\epsilon^{-2/3})$ iterations.
\end{theorem}
\begin{proof}
When $f$ is convex, $v_k$ in (\ref{eq:vk}) is given by
\begin{eqnarray*}
v_k = \nabla f(y_k) - \nabla f(x_k) + 4L[x_k - y_k].
\end{eqnarray*}
Hence,
\begin{eqnarray}\label{ineq:vk}
\| v_k \| \leq 5L \| y_k - x_k \|.
\end{eqnarray}
On the other hand, from (\ref{ineq:convexityfk10}) and (\ref{ineq:ak}), we have
\begin{eqnarray*}
&    & \frac{1}{16}L \left( \min_{n+1 \leq k \leq 2n}  \| y_k - x_k \|^2  \right) \sum_{k=n+1}^{2n} k^2 \leq  L \left( \min_{n+1 \leq k \leq 2n}  \| y_k - x_k \|^2  \right) \sum_{k=n+1}^{2n} a_{k-1}^2 \\
& \leq   & (a_{n-1}^2[\phi(y_{n}) - \phi(y^\ast)] + 2L\| a_{n-1}[y_{n} - y_{n-1}] + y_{n-1} - y^\ast \|^2) - \\
& &  (a_{2n - 1}^2 [ \phi(y_{2n}) - \phi(y^\ast)] + 2L\| a_{2n -1}[y_{2n} - y_{2n - 1}] + y_{2n - 1} - y^\ast \|^2),
\end{eqnarray*}
where the right hand side in the above inequality tends to zero as $n$ tends to infinity since $\{a_{k-1}^2[\phi(y_k) - \phi(y^\ast)] + 2L \|a_{k-1}[y_k - y_{k-1}] + y_{k-1} - y^\ast \|^2 \}$ is a convergent sequence.  Hence, we have
\begin{eqnarray*}
\min_{n+1 \leq k \leq 2n}  \| y_k - x_k \|^2  = o(n^{-3}).
\end{eqnarray*}
The above equality together with (\ref{ineq:vk}) imply that
\begin{eqnarray*}
\min_{n + 1 \leq k \leq 2n} \| v_k \| = o(n^{-3/2}).
\end{eqnarray*}
The theorem is hence proved.
\end{proof}

\vspace{10pt}

\noindent The above theorem takes existing non-asymptotic results in the literature, as appeared in \cite{Ghadimi,Liang3} (see also \cite{Shi}), and in Theorem \ref{thm:main}, where we have an iteration complexity to find an $\epsilon$-approximate solution to Problem (\ref{CNO}) of $\mathcal{O}(\epsilon^{-2/3})$, a step forward by stating an asymptotic convergence rate result - at most $o(\epsilon^{-2/3})$ iterations to find an $\epsilon$-approximate solution to Problem (\ref{CNO}).  

\vspace{10pt}

\noindent In the general case when $f$ can be nonconvex, we have the following result:
\begin{theorem}\label{thm:main3}
Suppose $\{y_k\}$ is convergent.  Then Algorithm mFISTA finds an $\epsilon$-approximate solution to Problem (\ref{CNO}) in at most $o(\epsilon^{-2})$ iterations.
\end{theorem}
\begin{proof}
Suppose $y_k \rightarrow y^{\ast\ast}$ as $k \rightarrow \infty$.  It is easy to check that Lemma \ref{lem:keyinequality}, in particular (\ref{ineq:convexityfk7}), still holds with $y^\ast$ replaced by $y^{\ast\ast}$.  Hence, we have for $k \geq 2$,
\begin{eqnarray*}
& & a_{k-1}^2[\phi(y_{k}) - \phi(y^{\ast\ast})] - a_{k-2}^2 [\phi(y_{k-1}) - \phi(y^{\ast\ast})] + \frac{L_k}{2}a_{k-1}^2\|y_{k} - y_{k-1} \|^2 +  L a_{k-1}^2 \| y_{k} - {x}_{k} \|^2 \\
& \leq &   2 L [\|a_{k-2}[ y_{k-1} - y_{k-2}]  + y_{k-2} - y^{\ast\ast} \|^2 - \| a_{k-1}[ y_{k} - y_{k-1}]  + y_{k-1} - y^{\ast\ast} \|^2] +\\
&   & \frac{\underline{m} - L_k}{2} a_{k-1}  \| y^{\ast\ast} - y_{k} \|^2 + \frac{L_k}{2} a_{k-1} \| y_{k-1} - y^{\ast\ast} \|^2. 
\end{eqnarray*}
Summing the above inequality from $k = n+1$ to $2n$ leads to
\begin{eqnarray}
&   & \frac{1}{2} \sum_{k=n+1}^{2n} L_k a_{k-1}^2 \| y_k - y_{k-1} \|^2 + L \sum_{k=n+1}^{2n} a_{k-1}^2 \| y_k - x_k \|^2 \nonumber \\
& \leq & a_{n-1}^2 [\phi(y_{n}) - \phi(y^{\ast\ast})] - a_{2n-1}^2[\phi(y_{2n}) - \phi(y^{\ast\ast})] \nonumber \\
&   & 2L[\|a_{n-1}^2[y_n - y_{n-1}] + y_{n-1} - y^{\ast\ast} \|^2 - \| a_{2n-1}^2[y_{2n} - y_{2n-1}] + y_{2n-1} - y^{\ast\ast}\|^2] + \nonumber \\
&  & \frac{1}{2} \sum_{k=n+1}^{2n} (\underline{m} - L_k) a_{k-1} \| y^{\ast\ast} - y_k \|^2 + \frac{1}{2} \sum_{k=n+1}^{2n} L_k a_{k-1} \| y_{k-1} - y^{\ast\ast}\|^2. \label{ineq:convexityfk11}
\end{eqnarray}
From (\ref{ineq:convexityfk11}), using (\ref{ineq:ak}) and (\ref{ineq:ak2}), and following similar arguments which appear earlier in the paper, it is easy to check that we have
\begin{eqnarray*}
&  & \left( \min_{n+1 \leq k \leq 2n}  \left( 36L \| x_k - y_k \|^2 + L_k \| y_k - y_{k-1} \|^2 \right) \right) \sum_{k=n+1}^{2n} k^2 \\
& \leq & K n^2 [ |\phi(y_n) - \phi(y^{\ast\ast})| + |\phi(y_{2n}) - \phi(y^{\ast\ast})| + \| y_n - y_{n-1} \|^2 + \|y_{2n} - y_{2n-1} \|^2 + \\
&   & \max_{n \leq k \leq 2n} \| y_{k} - y^{\ast\ast} \|^2],
\end{eqnarray*}
where $K$ is some positive constant independent of $n$.  Note that 
\begin{eqnarray*}
 |\phi(y_n) - \phi(y^{\ast\ast})| + |\phi(y_{2n}) - \phi(y^{\ast\ast})| + \| y_n - y_{n-1} \|^2 + \|y_{2n} - y_{2n-1} \|^2 +  \max_{n \leq k \leq 2n} \| y_{k} - y^{\ast\ast} \|^2
 \end{eqnarray*}
 tends to zero as $n$ tends to infinity, since $y_k \rightarrow y^{\ast\ast}$ as $k \rightarrow \infty$.   On the other hand, we have
 \begin{eqnarray*}
\min_{1 \leq k \leq 2n} \| v_k \| \leq 2 \sqrt{L}  \left( \min_{n+1 \leq k \leq 2n} \left( 36L \| x_k - y_k \|^2 + L_k \| y_k - y_{k-1} \|^2  \right) \right)^{1/2}
\end{eqnarray*} 
which can be derived in a similar way as  (\ref{ineq:main}).  Putting everything together, we obtain
\begin{eqnarray*}
\min_{1 \leq k \leq 2n} \| v_k \|  = o(n^{-1/2}),
\end{eqnarray*}
and the theorem is proved.
\end{proof}

\vspace{10pt}

\noindent The result in the above theorem contributes further to the literature complementing the paper \cite{Gratton} on first order algorithms to solve nonconvex optimization problems.

\section{Conclusion}\label{sec:conclusion}

In this paper, we propose a first order algorithm, Algorithm mFISTA, to solve an optimization problem whose objective function is the sum of a possibly nonconvex function, with Lipschitz continuous gradient, and a convex function which can be nonsmooth.  A feature of the algorithm is its ability to ``detect'' when the objective function is convex to obtain an improved iteration complexity.  We analyze the algorithm to provide iteration complexity results for the algorithm in the nonconvex and convex case.  We further provide asymptotic convergence rate results for the algorithm also in the nonconvex and convex case that take the iteration complexity results for these cases a step forward.

\vspace{10pt}

\noindent {\bf{\underline{Data Availability Statement}}}

\vspace{10pt}

\noindent No datasets were generated or analyzed during the current study.

\bibliographystyle{plain}
\bibliography{Reference_Sim}


\end{document}